\documentclass[a4paper]{amsart}

\title[A note on stochastic integrals as $L^2$-curves]{A note on stochastic integrals as $L^2$-curves}
\author{Stefan Tappe}
\address{ETH Z\"urich, Department of Mathematics, R\"amistrasse 101, CH-8092 Z\"urich, Switzerland}
\email{stefan.tappe@math.ethz.ch}

\usepackage{amscd}
\usepackage{amsmath}
\usepackage{amssymb}
\usepackage{amsthm}
\usepackage{bbm}
\usepackage{stmaryrd}

\usepackage{natbib}

\newif\ifpdf
\ifx\pdfoutput\undefined
   \pdffalse        
\else
   \pdfoutput=1     
   \pdftrue
\fi

\ifpdf
   \usepackage[pdftex]{graphicx}
\else
   \usepackage{graphicx}
\fi

\numberwithin{equation}{section} \swapnumbers

\newtheorem{satz}{Satz}[section]

\newtheorem{proposition}[satz]{Proposition}

\newtheorem{lemma}[satz]{Lemma}

\newtheorem{definition}[satz]{Definition}

\begin{document}

\maketitle\thispagestyle{empty}

\begin{abstract}
In a work of \cite{Onno-Levy} stochastic integrals are
regarded as $L^2$-curves. In \cite{Filipovic-Tappe} we have shown
the connection to the usual It\^o-integral for
c\`adl\`ag-integrands. The goal of this note is to complete this
result and to provide the full connection to the It\^o-integral. We
also sketch an application to stochastic partial differential
equations.
\bigskip

\end{abstract}


\section{Introduction}

In the paper \cite{Filipovic-Tappe} we have established an existence
and uniqueness result for stochastic partial differential equations,
driven by L\'evy processes, by applying a result from
\cite[Thm. 4.1]{Onno-Levy}. In \cite{Onno-Levy} stochastic integrals
are regarded as $L^2$-curves. It was therefore necessary to
establish the connection to the usual It\^o-integral (developed,
e.g., in \cite{Jacod-Shiryaev} or \cite{Protter}) for
c\`adl\`ag-integrands, which we have provided in \cite[Appendix
B]{Filipovic-Tappe}.

The goal of the present note is to complete this result and to
provide the full connection to the It\^o-integral. More precisely,
we will show that the space of adapted $L^2$-curves is embedded into
the space of It\^o-integrable processes (see Proposition
\ref{prop-embedding} below), and that the corresponding
It\^o-integral is a c\`adl\`ag-version of the stochastic integral in
the sense of \cite{Onno-Levy} (see Proposition \ref{prop-coin}
below).

This is the content of Section \ref{sec-curves}. Afterwards, we
outline an application to stochastic partial differential equations
in Section \ref{sec-SDE}.

\section{Stochastic integrals as $L^2$-curves}\label{sec-curves}

Throughout this text, let $(\Omega,\mathcal{F},(\mathcal{F}_t)_{t
\geq 0},\mathbb{P})$ be a filtered probability space satisfying the
usual conditions. Furthermore, let $(H,\| \cdot \|)$ denote a
separable Hilbert space.

For any $T \in \mathbb{R}_+$ the space $C[0,T] :=
C([0,T];L^2(\Omega;H))$ of all continuous curves from $[0,T]$ into
$L^2(\Omega;H)$ is a Banach space with respect to the norm
\begin{align*}
\| r \|_{T} := \sup_{t \in [0,T]} \| r_t \|_{L^2(\Omega;H)} = \sqrt{
\sup_{t \in [0,T]} \mathbb{E} [ \| r_t \|^2 ] }.
\end{align*}
The subspace $C_{\rm ad}[0,T]$ consisting of all adapted processes
from $C[0,T]$ is closed with respect to this norm. Note that, by the
completeness of the filtration $(\mathcal{F}_t)_{t \geq 0}$,
adaptedness is independent of the choice of the representative.

\subsection{Stochastic integral with respect to a L\'evy martingale}

Let $M$ be a real-valued, square-integrable L\'evy martingale. We
recall how in this case the stochastic integral $\text{{\rm
(G-)}}(\Phi \cdot M)$ in the sense of \cite[Sec.
3]{Onno-Levy} is defined for $\Phi \in C_{\rm ad}[0,T]$.

\begin{lemma}\label{lemma-ex-Levy-integral}
\cite[Prop. 3.2.1]{Onno-Levy} Let $\Phi \in C_{\rm ad}[0,T]$ be
arbitrary. For each $t \in [0,T]$ there exists a unique random
variable $Y_t \in L^2(\Omega)$ such that for every $\varepsilon > 0$
there exists $\delta > 0$ such that
\begin{align}\label{partition-vG-int}
\mathbb{E} \left[ \bigg\| Y_t - \sum_{i=0}^{n-1} \Phi_{t_i}
(M_{t_{i+1}} - M_{t_i}) \bigg\|^2 \right] < \varepsilon
\end{align}
for every partition $0 = t_0 < t_1 < \ldots < t_n = t$ with $\sup_{i
= 0,\ldots,n-1} |t_{i+1} - t_i| < \delta$.
\end{lemma}

\begin{definition}\label{def-Levy-integral}
(\cite{Onno-Levy}) Let $\Phi \in C_{\rm ad}[0,T]$ be arbitrary. Then
the stochastic integral $Y = \text{{\rm (G-)}}(\Phi \cdot M)$ is the
stochastic process $Y = (Y_t)_{t \in [0,T]}$ where every $Y_t$ is
the unique element from $L^2(\Omega)$ such that
(\ref{partition-vG-int}) is valid.
\end{definition}

We observe that the integrand $\Phi$ as well as the stochastic
integral $\text{{\rm (G-)}} (\Phi \cdot M)$ are only determined up
to a version. In particular, it is not clear if the integral process
has a c\`adl\`ag-version.

\begin{lemma}
\cite[Thm. 3.3.2]{Onno-Levy} For each $\Phi \in C_{\rm ad}[0,T]$ we
have
\begin{align*}
\text{{\rm (G-)}}(\Phi \cdot M) \in C_{\rm ad}[0,T].
\end{align*}
\end{lemma}

We are now interested in finding the connection between the
stochastic integral $\text{{\rm (G-)}}(\Phi \cdot M)$ and the usual
It\^o-integral (developed, e.g., in \cite{Jacod-Shiryaev}, \cite{Protter}, \cite{Applebaum} 
for the finite dimensional case and in \cite{Da_Prato},
\cite{P-Z-book} for the infinite dimensional case). We use the
abbreviation
\begin{align*}
L^2(\mathcal{P}_T) := L^2(\Omega \times
[0,T],\mathcal{P}_T,\mathbb{P} \otimes \lambda;H),
\end{align*}
where $\mathcal{P}_T$ denotes the predictable $\sigma$-algebra on
$\Omega \times [0,T]$ and $\lambda$ the Lebesgue measure. Since for
any square-integrable L\'evy martingale $M$ the predictable
quadratic covariation $\langle M,M \rangle$ is linear,
$L^2(\mathcal{P}_T)$ is the space of all $L^2$-processes $\Phi$, for
which the It\^o-integral $\Phi \cdot M$ exists, independent of the
choice of $M$.

\begin{lemma}
For each $\Phi \in C_{\rm ad}[0,T]$ there exists a predictable
version ${}^p \Phi \in L^2(\mathcal{P}_T)$ of $\Phi$.
\end{lemma}

\begin{proof}
By \cite[Prop. 3.6.ii]{Da_Prato} there exists a predictable version
${}^p \Phi$ of $\Phi$. Since $\Phi \in C_{\rm ad}[0,T]$, we also
have
\begin{align*}
\int_0^T \mathbb{E} [ \| {}^p \Phi_t \|^2 ] dt = \int_0^T \mathbb{E}
[ \| \Phi_t \|^2 ] dt \leq T \sup_{t \in [0,T]} \mathbb{E}[\| \Phi_t
\|^2] < \infty,
\end{align*}
that is ${}^p \Phi \in L^2(\mathcal{P}_T)$.
\end{proof}

\begin{proposition}\label{prop-embedding}
The map $\Phi \mapsto {}^p \Phi$ defines an embedding from $C_{\rm
ad}[0,T]$ into $L^2(\mathcal{P}_T)$.
\end{proposition}

\begin{proof}
For two predictable versions $\Phi^1,\Phi^2 \in L^2(\mathcal{P}_T)$
of $\Phi$ we have
\begin{align*}
\int_0^T \mathbb{E} [\| \Phi_t^1 - \Phi_t ^2 \|^2] dt = 0,
\end{align*}
whence $\Phi^1 = \Phi^2$ in $L^2(\mathcal{P}_T)$. Therefore, the map
$\Phi \mapsto {}^p \Phi$ is well-defined. The linearity of $\Phi
\mapsto {}^p \Phi$ is immediately checked, and the estimate
\begin{align*}
\int_0^T \mathbb{E}[\| {}^p \Phi_t \|^2] dt = \int_0^T \mathbb{E}[\|
\Phi_t \|^2] dt \leq T \sup_{t \in [0,T]} \mathbb{E}[\| \Phi_t
\|^2], \quad \Phi \in C_{\rm ad}[0,T]
\end{align*}
proves the continuity of $\Phi \mapsto {}^p \Phi$. For $\Phi \in
C_{\rm ad}[0,T]$ with ${}^p \Phi = 0$ in $L^2(\mathcal{P}_T)$ we
have
\begin{align*}
\int_0^T \mathbb{E}[\| \Phi_t \|^2] dt = \int_0^T \mathbb{E}[\| {}^p
\Phi_t \|^2] dt = 0.
\end{align*}
Since $\Phi \in C_{\rm ad}[0,T]$, the map $t \mapsto \mathbb{E}[\|
\Phi_t \|^2]$ is continuous, which implies $\Phi = 0$ in $C_{\rm
ad}[0,T]$, showing that $\Phi \mapsto {}^p \Phi$ is injective.
\end{proof}

The notation ${}^p \Phi$ reminds of the {\em predictable projection}
of a process $\Phi$, which we shall briefly recall. In the
real-valued case one defines, for every
$\overline{\mathbb{R}}$-valued and $\mathcal{F}_T \otimes
\mathcal{B}[0,T]$-measurable process $\Phi$ the predictable
projection ${}^{\pi} \Phi$ of $\Phi$, according to \cite[Thm.
I.2.28]{Jacod-Shiryaev}, as the (up to an evanescent set) unique
$(-\infty,\infty]$-valued process satisfying the following two
conditions:
\begin{enumerate}
\item It is predictable;

\item $({^{\pi}} \Phi)_{\tau} = \mathbb{E}[\Phi_{\tau} \, | \, \mathcal{F}_{\tau
-}]$ on $\{ \tau \leq T \}$ for all predictable times $\tau$.
\end{enumerate}
Note that for every predictable process $\Phi$ we have ${}^{\pi}
\Phi = \Phi$.

We transfer this definition to any $H$-valued process
$\tilde{\Phi}$, which is an $\mathcal{F}_T \otimes
\mathcal{B}[0,T]$-measurable version of a process $\Phi \in
C_{\rm ad}[0,T]$ by using the notion of conditional expectation from
\cite[Sec. 1.3]{Da_Prato}. Then, the second property of the
predictable projection ensures that ${}^{\pi} \tilde{\Phi}$ is
finite, i.e. $H$-valued.

We obtain the following relation between the embedding ${}^p \Phi$
and the predictable projection ${}^{\pi} \tilde{\Phi}$:

\begin{lemma}\label{lemma-pred-proj}
For each $\Phi \in C_{\rm ad}[0,T]$ and every $\mathcal{F}_T \otimes
\mathcal{B}[0,T]$-measurable version $\tilde{\Phi}$ we have
\begin{align*}
{}^{\pi} \tilde{\Phi} = {}^p \Phi \quad \text{in
$L^2(\mathcal{P}_T)$.}
\end{align*}
\end{lemma}

\begin{proof}
For each $t \in [0,T]$ the identities
\begin{align*}
{}^{\pi} \tilde{\Phi}_t = \mathbb{E}[ \tilde{\Phi}_t \, | \,
\mathcal{F}_{t-}] = \mathbb{E}[ \Phi_t \, | \, \mathcal{F}_{t-} ] =
\mathbb{E}[ {}^p \Phi_t \, | \, \mathcal{F}_{t-} ] = {}^p \Phi_t
\quad \text{$\mathbb{P}$--a.s.}
\end{align*}
are valid, which gives us
\begin{align*}
\int_0^T \mathbb{E}[\|{}^{\pi} \tilde{\Phi}_t - {}^p \Phi_t\|^2]dt =
0,
\end{align*}
proving the claimed result.
\end{proof}

\begin{lemma}\label{lemma-cadlag}
If $\Phi \in C_{\rm ad}[0,T]$ has a c\`adl\`ag-version, then we have
\begin{align*}
{}^p \Phi = \Phi_- \quad \text{in $L^2(\mathcal{P}_T)$.}
\end{align*}
\end{lemma}

\begin{proof}
The process $\Phi_-$ is predictable and we have
\begin{align*}
\mathbb{E} \bigg[ \int_0^T \| {}^p \Phi_t - \Phi_{t-} \|^2 dt \bigg] =
\mathbb{E} \bigg[ \int_0^T \| \Phi_t - \Phi_{t-} \|^2 dt \bigg] =
\mathbb{E} \bigg[ \int_0^T \| \Delta \Phi_t \|^2 dt \bigg] = 0,
\end{align*}
because $\mathcal{N}_{\omega} = \{ t \in [0,T] : \Delta
\Phi_t(\omega) \neq 0 \}$ is countable for all $\omega \in \Omega$.
\end{proof}

\begin{proposition}
For each $\Phi \in C_{\rm ad}[0,T]$ we have
\begin{align*}
\text{{\rm (G-)}}(\Phi \cdot M) = {}^p \Phi \cdot M \quad \text{in
$C_{\rm ad}[0,T]$.}
\end{align*}
In particular, $\text{{\rm (G-)}}(\Phi \cdot M)$ has a
c\`adl\`ag-version.
\end{proposition}

\begin{proof}
Let $t \in [0,T]$ and $\epsilon > 0$ be arbitrary. Since $\Phi \in
C_{\rm ad}[0,T]$, it is uniformly continuous on the compact interval
$[0,t]$, and thus there exists $\delta > 0$ such that
\begin{align}\label{continuity-e-d}
\mathbb{E}[\| \Phi_u - \Phi_v \|^2] < \frac{\epsilon}{\langle M,M
\rangle_t}
\end{align}
for all $u,v \in [0,t]$ with $|u - v| < \delta$. Let $\mathcal{Z} =
\{ 0 = t_0 < t_1 < \ldots < t_n = t\}$ be an arbitrary decomposition
with $\sup_{i = 0,\ldots,n-1} |t_{i+1} - t_i| < \delta$. Defining
\begin{align*}
\Phi^{\mathcal{Z}} := \Phi_0 \mathbbm{1}_{[0]} + \sum_{i = 0}^{n-1}
\Phi_{t_i} \mathbbm{1}_{(t_i,t_{i+1}]},
\end{align*}
we obtain, by using the It\^o-isometry and (\ref{continuity-e-d}),
\begin{align*}
&\mathbb{E} \left[ \bigg\| ({}^p \Phi \cdot M)_t - \sum_{i=0}^{n-1}
\Phi_{t_i} (M_{t_{i+1}} - M_{t_i}) \bigg\|^2 \right] = \mathbb{E}
\left[ \bigg\| \int_0^t ({}^p \Phi_s  - \Phi_s^{\mathcal{Z}}) dM_s
\bigg\|^2 \right]
\\ &= \mathbb{E} \bigg[ \int_0^t \| {}^p \Phi_s - \Phi_s^{\mathcal{Z}} \|^2 d \langle
M,M \rangle_s \bigg] = \sum_{i=0}^{n - 1} \int_{t_i}^{t_{i+1}}
\mathbb{E} [ \| \Phi_s - \Phi_{t_i}\|^2 ] d\langle M,M \rangle_s <
\epsilon,
\end{align*}
establishing that ${}^p \Phi \cdot M$ is a version of $\text{{\rm
(G-)}}(\Phi \cdot M)$.
\end{proof}

\subsection{Stochastic integral with respect to Lebesgue measure}

In an analogous fashion, we introduce the stochastic integral ${\rm
\text{(G-)}}(\Phi \cdot \lambda)$ with respect to the Lebesgue
measure $\lambda$, cf. \cite[Lemma 3.6]{Onno-Levy}. By similar
arguments as in the previous subsection, we obtain the same relation
between this stochastic integral ${\rm \text{(G-)}}(\Phi \cdot
\lambda)$ and the usual Bochner integral $\Phi \cdot \lambda$.

\subsection{Stochastic integral with respect to a L\'evy process}\label{subsec-Levy}

Now let $X$ be a square-integrable L\'evy process with
semimartingale decomposition $X_t = M_t + bt$, where $M$ is a
square-integrable L\'evy martingale and $b \in \mathbb{R}$.
According to \cite[Def. 3.7]{Onno-Levy} we set
\begin{align*}
{\rm \text{(G-)}} (\Phi \cdot X) := \text{{\rm (G-)}} (\Phi \cdot M)
+ b \text{{\rm (G-)}} (\Phi \cdot \lambda).
\end{align*}
As a direct consequence of our previous results, we obtain:

\begin{proposition}\label{prop-coin}
For each $\Phi \in C_{\rm ad}[0,T]$ we have
\begin{align*}
\text{{\rm (G-)}}(\Phi \cdot X) = {}^p \Phi \cdot X \quad \text{in
$C_{\rm ad}[0,T]$.}
\end{align*}
In particular, $\text{{\rm (G-)}}(\Phi \cdot X)$ has a
c\`adl\`ag-version.
\end{proposition}

Summing up, we have seen that the space $C_{\rm ad}[0,T]$ of all
adapted continuous curves from $[0,T]$ into $L^2(\Omega;H)$ is embedded into
$L^2(\mathcal{P}_T)$ via $\Phi \mapsto {}^p \Phi$, see Proposition
\ref{prop-embedding}, and that the It\^o-integral ${}^p \Phi \cdot
X$ is a c\`adl\`ag-version of $\text{{\rm (G-)}}(\Phi \cdot X)$, see
Proposition \ref{prop-coin}. Moreover, we have seen the relation to
the predictable projection in Lemma \ref{lemma-pred-proj}.

We close this section with an example, which seems surprising at a
first view. Let $X$ be a standard Poisson process with values in
$\mathbb{R}$. In Ex. 3.9 in \cite{Onno-Levy} it is derived that
\begin{align*}
{\rm \text{(G-)}}\int_0^t X_s dX_s = \frac{1}{2} ( X_t^2 - X_t ).
\end{align*}
Apparently, this does not coincide with the pathwise
Lebesgue-Stieltjes integral
\begin{align*}
\int_0^t X_s dX_s = \frac{1}{2} ( X_t^2 + X_t ).
\end{align*}
The explanation for this seemingly inconsistency is easily provided.
The process $X$ is not predictable, whence it is not
It\^o-integrable, and a straightforward calculation shows that
\begin{align*}
{\rm \text{(G-)}}\int_0^t X_s dX_s = \int_0^t X_{s-} dX_s.
\end{align*}
This, however, is exactly what an application of Proposition
\ref{prop-coin} and Lemma \ref{lemma-cadlag} yields.

\section{Solutions of stochastic partial differential
equations as $L^2$-curves}\label{sec-SDE}

Regarding stochastic integrals as $L^2$-curves provides an existence
and uniqueness proof for stochastic partial differential equations.
Of course, this result is well-known in the literature (see, e.g.,
\cite{Ruediger-mild}, \cite{Da_Prato}, \cite{SPDE}, \cite{Marinelli-Prevot-Roeckner}, \cite{P-Z-book}), whence we only give an outline.

Consider the stochastic partial differential equation
\begin{align}\label{SPDE}
\left\{
\begin{array}{rcl}
dr_t & = & (A r_t + \alpha(t,r_t))dt + \sum_{i=1}^n
\sigma_i(t,r_{t-})dX_t^i \medskip
\\ r_0 & = & h_0,
\end{array}
\right.
\end{align}
where $A : \mathcal{D}(A) \subset H \rightarrow H$ denotes the
infinitesimal generator of a $C_0$-semigroup $(S_t)_{t \geq 0}$ on
$H$, and where $X^1,\ldots,X^n$ are real-valued, square-integrable
L\'evy processes as in Section \ref{subsec-Levy}. We assume that the standard Lipschitz conditions
are satisfied.

Then, there exists a unique solution $r \in C_{\rm ad}[0,T]$ of the equation
\begin{align*}
r_t := S_t h_0 + \text{{\rm (G-)}} \int_0^t S_{t-s} \alpha(s,r_s)ds +
\sum_{i=1}^n \text{{\rm (G-)}} \int_0^t S_{t-s}
\sigma_i(s,r_s)dX_s^i,
\end{align*}
see \cite{Onno-Wiener} for the Wiener case and \cite{Onno-Levy} for
the L\'evy case. It is remarkable that the proof is established by
means of precisely the same arguments as in the classical
Picard-Lindel\"of iteration scheme for ordinary differential
equations, where one works on the Banach space $C([0,T];H)$ instead
of $C_{\rm ad}[0,T]$.

Applying Proposition \ref{prop-coin} for any fixed $t \in [0,T]$, we
obtain the existence of a (up to a version) unique, predictable mild
solution for the SPDE
\begin{align*}
\left\{
\begin{array}{rcl}
dr_t & = & (A r_t + \alpha(t,r_t))dt + \sum_{i=1}^n
\sigma_i(t,r_{t})dX_t^i
\medskip
\\ r_0 & = & h_0,
\end{array}
\right.
\end{align*}
which, in addition, is mean-square continuous.

Observe that we have no statement on path properties of the
solution. If, however, the semigroup in pseudo-contractive, i.e.,
there exists $\omega \in \mathbb{R}$ such that
\begin{align*}
\| S_t \| \leq e^{\omega t}, \quad t \geq 0
\end{align*}
then the stochastic convolution (It\^o-)integrals have a
c\`adl\`ag-version. This can be shown by using the Kotelenez
inequality (see \cite{Kotelenez}) or by using the
Sz\H{o}kefalvi-Nagy theorem on unitary dilations (see, e.g.,
\cite[Thm. I.8.1]{Nagy}, or \cite[Sec. 7.2]{Davies}). We refer to
\cite[Sec. 9.4]{P-Z-book} for an overview. In this case, we conclude
that there even exists a (up to indistinguishability) unique
c\`adl\`ag, adapted mild solution $(r_t)_{t \geq 0}$ for
(\ref{SPDE}), which, in addition, is mean-square continuous.

\subsection*{Acknowledgement}

The author gratefully acknowledges the support from WWTF (Vienna Science and Technology Fund).

The author is also grateful to an anonymous referee for her/his helpful comments and suggestions.

\bibliographystyle{apalike}

\bibliography{Integrals_Curves}

\end{document}